\documentclass[12pt]{amsart}
\usepackage{amsmath}
\usepackage{graphicx}
\usepackage{epstopdf}
\usepackage{amssymb,amscd,array}
\usepackage[pagewise]{lineno}
\usepackage{color}
\usepackage{float}
\usepackage{enumerate}
\usepackage{tasks}

\usepackage{amsmath,amssymb,amsbsy,amsfonts,latexsym,amsopn,amstext,cite,
	amsxtra,euscript,amscd,bm}

\makeindex \makeatletter
\def\captionof#1#2{{\def\@captype{#1}#2}}
\makeatother

\usepackage{mathtools}
\usepackage{todonotes}
\usepackage{url}

\newcounter{tablegroup}
\newcounter{subtable}[tablegroup]

\newtheorem{thm}{Theorem}[section]
\newtheorem{cor}[thm]{Corollary}
\newtheorem{lem}[thm]{Lemma}
\newtheorem{prop}[thm]{Proposition}

\numberwithin{equation}{section}
\theoremstyle{remark}
\newtheorem*{rem}{\bf Remarks}
\newtheorem*{que}{\bf Question}

\newcommand{\Real}{\mathbb R}
\newcommand{\eps}{\varepsilon}

\newcommand{\Or}[1]{\textrm{Orb}_{G}(#1)}
\newcommand{\F}{\mathbb{F}}
\newcommand{\tend}[3][]{\xrightarrow[#2\to#3]{#1}}

\newcommand{\ds}{\displaystyle}

\newcommand{\Z}{\mathbb{Z}}

\newcommand{\C}{\mathbb{C}}
\newcommand{\Co}{\mathcal{C}}

\newcommand{\Pa}{\mathcal{P}}

\newcommand{\diam}{\textrm{diam}}

\title[ Actions with finite orbits on local dendrites]
{Actions with finite orbits on local dendrites}

\author{e. H. e. Abdalaoui and I. Naghmouchi}
\date{\today}

\address{el Houcein el Abdalaoui, University of Rouen Normandy,
	Department of Mathematics, LMRS  UMR 6085 CNRS, Avenue de l'Universit\'e, BP.12,
	76801 Saint Etienne du Rouvray - France.}
\email{elhoucein.elabdalaoui@univ-rouen.fr}

\address{ Issam Naghmouchi, University of Carthage, Faculty of Science of Bizerte, (UR17ES21), Dynamical systems and their applications, 7021, Jarzouna, Tunisia.}
\email{issam.naghmouchi@fsb.rnu.tn and issam.nagh@gmail.com}


\begin{document}
\begin{abstract} It is shown that the restriction of the action of any group with finite orbit on the minimal sets of local dendrites is equicontinuous. Consequently, we obtain that the action of any amenable group and Thompson group restricted to any minimal sets of dendrite is equicontinuous. We further provide a class of non-amenable groups whose action on the minimal sets of local dendrites is equicontinuous. Moreover, we extend some of our results to dendron. We further give a characterization of the set of invariant probability measures and its extreme points.
\end{abstract}
\maketitle

\section{\bf Introduction.}

This paper deals with the action of group on local dendrites. It turns out that the minimal action on the nondegenerete dendrite is similar to the minimal action on the circle. For the group action on the circle, Margulis proved that either the group is not amenable or there is a G-invariant probability measure \cite{Marg}. It is also well-known that the minimal sets can only be the whole circle, or a finite set, or a  \linebreak Cantor set \cite{Beklaryan}. Moreover, the minimal action  on the circle is either equicontinuous or strongly proximal, and if the later case holds then the group must contains a free non-commutative subgroup
 \cite{Marg}. Nowadays there are several proofs of this theorem, see \cite[Section 5.2]{Ghys}, \cite[Chap.2, p. 54]{Navas}, \cite[Theorem 4.5]{BT}.\\

For dendrite, H. Marzougui and the second author classify the minimal sets of the group action on dendrite \cite{M-Nag} and local dendrites \cite{M-Nag2}. This classification is analogue to the case of the circle.\\

Exploiting result from \cite{M-Nag}, E. H. Shi and X. D. Ye proved that every amenable group action on dendrites has a fixed point or finite orbit of order $2$ \cite{SYO}.  Very recently, E. Glasner and M. Megrelishvili proved that  every continuous group action of $G$ on a dendron is tame \cite{Gla-M}. They asked
\begin{que}	Is there an amenable group $G$, an action of $G$ on a dendrite $X$, and a minimal subset $Y \subset X$, such that the system $(G,Y)$ is almost automorphic but not equicontinuous?
\end{que}

The answer to this question was given by E. H. Shi and X. D. Ye in \cite{SYR}. They proved that the restriction action of any amenable group on any minimal set $K$ is equicontinuous. Their proof is  based on their result mentioned above.\\

In this paper, we will strengthened this result by considering the action on a local dendrite and by relaxing the amenability condition.\\

Indeed, we will established that if the group action admit a finite orbit then its restriction on any minimal set $K$ is equicontinuous. As a consequence, we get that the  restriction of Thompson group on any minimal set $K$ of local dendrite is equicontinuous.  We further obtain the result from \cite[Corollary 6.7]{M-Nag}.\\

Moreover, we will proved that there is a class of non-amenable groups for which the infinite minimal sets are a Cantor and the action has a finite orbit. Finally, we will describe the set of invariant probability measures and its extreme points.\\

Our proof is based on the descriptions of the structure of a minimal set and its convex hull and how finite orbits occur on its convex hull.\\

According to our results, it follows that the equicontnuity of the action on the minimal sets of local dendrite can not separate the amenable and not amenable groups. Nevertheless, we hope that the strategy of searching for the good topological dynamical property will serve to enlighten the  problem of amenability of Thompson group.\\

The plan of this paper is as follows. In Section \ref{tool}, we recall some definitions and tools on dendrites and local dendrites which are useful for the rest of the paper. In section \ref{SMain-1}, we state our main results and we give the proof of our main first result. In section \ref{SMain-2}, we recall some basic definitions and tools on non-amenable groups and Thompson group $F$, we further give the proof of our second and third main results. Finally, in section \ref{SMain-3}, we describe the set of invariant measures and its extreme points.


\section{Set up and tools.}\label{tool}
Let $G$ be a group acting continuously on the topological space $X$, that is, there is a family $(T_g)_{g \in G}$ of continuous maps from $X$ to $X$ such that for any $g,g' \in G$, we have $T_g \circ T_{g'}=T_{gg'},$ and $T_e=Id_X$ where $e$ is the identity element of $G$ and $Id_X$ is the identity map on $X$. Obviously, for each $g\in G$, $T_g$ is an element of the group of homeomorphism on $X$ denoted by $Hom(X)$. For any $x\in X$, the subset
$\textrm{Orb}_{G}(x) = \{T_g(x): g \in G\}$ is called the \textit{orbit} of $x$ (under $G$).
A subset $A\subset X$ is called \textit{$G-$invariant} (resp. strongly $G-$invariant)
if $T_g(A)\subset A,$ for each $g \in G$ (resp., $T_g(A) = A$). It is called \textit{a minimal set of $G$} if it is
non-empty, closed, $G$-invariant and minimal (in the sense of
inclusion) for these properties, this is equivalent to say that it is an orbit closure that contains no smaller one;
for example a single finite
orbit. When $X$ itself is a minimal set, then we say that $G$ act \textit{minimally}. As usual, we denote the closure of any subset $A$  by $\overline{A}$. The orbit of a point $x \in X$ is said to be finite if $\overline{\Or{x})}=\Or{x}$ is a finite set.  Obviously, if a point $x$ is an atom of an invariant probability measure then its orbit is finite.  The action is said to have \textit{a finite orbit} if it admit a finite orbit.\\

For the case of action on dendrite, it is straightforward to see that the action has a finite orbit if and only if there is an invariant measure on $X$. Notice further that in this case the action fixed an arc. We shall strengthened this result to the action of group on dendron. \\

The point $x$ is said to be a recurrent point if there exists a net $(g_i)$ in $G$ such that $\lim g_i.x \rightarrow x$ and $g_i(x)\neq x$ for all $i$. \\

The action is proximal if for every $(x,y)\in X^2$ there is a net $(s_i)$ in $G$ and point $z \in X$ such that $lim s_i.x=\lim s_i.y=z$. It is well known that if the action is proximal then there is unique minimal set.\\

Let $\Co(X)$ be the space of continuous functions on $X$ equipped with the weak-star topology. By the classical representation theorem of Riesz the dual of $\Co(x)$ denoted by $\Co^*(X)$ is the space of the Borel bounded measures on $X$. Let $\Pa(X)$ be the set of probabilities measures on $X$. The elementary probability measures are given the Dirac measure on a point. We denote the Dirac measure on a point $x$ by $\delta_x$. By the theorem of Banach-Alaoglu-Bourbaki, $\Pa(X)$ is a compact convex set of the space of measure on $X$. The action of $G$ on $X$ induce an affine action on the $\Pa(X)$ as follows:
$$(g,\mu) \mapsto g\mu,$$
where $g\mu$ is the push-forward measure given by
$$g\mu(f)=\int f(g.x) d\mu(x),~~~~~~~~~~~~\forall f \in \Co$$
If the action of $G$ on $\Pa(X)$ is proximal we say that the action $G$ on $X$ is strongly proximal.  The compact set of $G$-invariant measures is denoted by $\Pa(X \looparrowleft G)$. We recall that $\mu$ is in $\Pa(X \looparrowleft G)$ if for any Borel set $A$, $\mu(g.A)=\mu(A).$ \\

Let $G$ acts on two compact space $X,Y$ and assume that there is a  subjective homeomorphism  $X$ to $Y$ such that
$$\phi \circ T_g=T_g \circ \phi,$$
then $Y$ is said to be a factor of $X$. By Lemma  3.7 from \cite{BT}, the factor of proximal action is proximal, so is for the factor of strongly proximal action.\\

Let us recall now the definitions and some basic properties of dendrites and local dendrites.\\

A \emph{continuum} is a compact connected space. Following \cite{Ward}, for any topological property $\mathcal{P}$, a continuum is said to be \textit{rim-}{$\mathcal{P}$} if it has a basis of open sets with boundaries enjoy property $\mathcal{P}$. For the rim-finite case the space are also called regular space \cite[Chap. VI, $\S{}$ 51, p. 274]{Kur}. Please, notice that therein the notion of order of the point is given with respect to the cardinality of the boundaries . An \emph{arc} is
any space homeomorphic to the compact interval $[0,1]$. A topological space
is \emph{arcwise connected} if any two of its points can be joined by an
arc.\\

By a \textit{graph} $X$, we mean a continuum which can be written as the union of finitely many arcs
such that any two of them are either
disjoint or intersect only in one or both of their endpoints.
%
\medskip

By a \textit{dendrite} $X$, we mean a locally connected metrizable continuum
containing no homeomorphic copy to a circle. Every sub-continuum of a
dendrite is a dendrite (\cite{nadler}, Theorem 10.10) and every
connected subset of $X$ is arc-wise connected (\cite{nadler},
Proposition 10.9). In addition, any two distinct points $x,y$ of a
dendrite $X$ can be joined by a unique arc with endpoints $x$ and
$y$, denote this arc by $[x,y]$ and let denote by
$[x,y) = [x,y]\setminus\{y\}$ (resp. $(x,y] = [x,y]\setminus\{x\}$ and
$(x,y) = [x,y]\setminus\{x,y\}$). A point $x\in X$ is called an
\textit{endpoint} if $X\setminus\{x\}$ is connected. It is called a
\textit{branch point} if $X\setminus \{x\}$ has more than two
connected components. The number of connected components of $X\setminus \{x\}$ is called the \textit{order} of $x$ and
denoted by ord$(x)$. The order of $x$ relatively to a subdendrite $Y$ of $X$ is denoted by $ord_Y(x)$.\\

Let us denote again by $E(X)$ and $B(X)$ the sets of
endpoints, and branch points of $X$, respectively.
By (\cite{Kur}, Theorem 6, 304 and Theorem 7, 302), $B(X)$ is at most countable. A point $x\in
X\setminus E(X)$ is called a \textit{cut point}. It is known that the  set of cut
points of $X$ is dense in $X$ (\cite{Kur}, VI, Theorem 8, p. 302).
Following (\cite{Ar}, Corollary 3.6), for any dendrite $X$, we have
B($X)$ is discrete whenever E($X)$ is closed.  An arc $I$ of $X$ is called \emph{free} if $I \cap B(X)=\emptyset$.
For a subset $A$ of $X$, we call \emph{the convex hull} of $A$, denoted by $[A]$, the intersection of all
sub-continua of $X$ containing $A$, one can write $[A] = \cup_{x, y\in A}[x,y]$.\\

We further have that $X$ is a dendrite if and only if any two
points of $X$ are separated in $X$ by a third point of $X$ (\cite{nadler},
Theorem 10.2). We recall that a point $z$ separates $x$ and $y$ in $X$ if there exist in $X$
open disjoint neighborhoods $U$, $V$ of $x$ and $y$ respectively such that $X \setminus \{z\} = U \cup V .$\\

More generally, a continuum $X$ is said to be a {\textit{dendron}} if every pair of distinct points $x,y$ can be separated in $X$ by a third point $z$. Obviously, a dendrite is a metrizable  dendron. A local dendron is a continuum having the property that every of its points has
a neighborhood which is a dendron. A \textit{local dendrite} is a metrizable local dendron. It is easy to see that the dendron is rim-finite. More generally, we have that any local dendron is rim-finite \cite{Ward}). Whence any local dendrite is rim-finite (for a direct proof, we refer to \cite[Theorem 1. page 303]{Kur}). We further have that any local dendrite is a locally connected continuum which contains at most
a finite number of simple closed curves (\cite[Theorem 4. page 304]{Kur}). It follows that every sub-continuum of a local dendrite is a local dendrite. Moreover, every graph and every dendrite is a local dendrite.\\

For a local dendrite $X$ containing at least one circle, let us denote by $\Gamma_X$ the minimal graph (in the sense of inclusion) which contain all the circles in $X$ (i.e.
the intersection of all graphs in $X$ that contain all the circles in $X$).\\

If $Y$ is a sub-continuum of a dendron $X$, then the canonical retraction of $X$ onto $Y$ is denoted by $r_Y$, for more details see \cite{dendron} (see also Theorem 3.12 in \cite{Gla-M}). If $a$ and $b$ are two distinct points in a dendron $X$, then we define the \textit{generalized arc} as follows:
$$[a,b]=\{x\in X: \ x \ \textrm{separated} \ a \textrm{~from~} \ b \textrm{~in~} \ X \}\cup \{u,v\}.$$

The convex hull of a given subset in a dendron is defined and denoted similarly as in the case of dendrite.\\

Let us also recall that the Vietoris topology is defined on the non-empty closed sets of $X$ denoted by $2^X$. It turns out $2^X$ equipped with Vietoris topology is compact and metrizable since $X$ is a compact metric space \cite{M-Nad}. Furthermore, the  metric is given by $$D(F_1,F_2)=\max\Big\{\sup_{x \in F_1}d(x,F_2),\sup_{y \in F_2}d(y,F_1)\Big\}, $$
where $d(x,F)=\inf_{y \in F}d(x,y).$
The metric $D$ is known as Hausdorff metric.\\

We notice further that the map $A \mapsto [A]$ is continuous with respect to the Vietoris topology \cite{Dush-Mono} when $X$ is a dendrite.\\

For a family of open set $\mathcal{U}$, we define
$\textrm{mesh}(\mathcal{U})$ by
$$\textrm{mesh}(\mathcal{U})=\sup\Big\{diam(U)/ U \in \mathcal{U} \Big\}.$$

Using the compactness of the Vietoris topology, it can be seen without using Zorn lemma that $X$ admit a  minimal set \cite[p.30]{Katok}. On the other hand, by applying Zorn lemma, it is easy to see the following.

\begin{prop}\label{sous-dendrite}Let $G$ a group acting on the dendrite $X$. Then, there exists a minimal $G$-invariant subdendrite $Y$.
\end{prop}

We warn the reader here that the minimality is used in the sense of inclusion in the family of all $G$-invariant subdendrites. Indeed, it is well known that there is example of minimal subdendrite which is not minimal in the usual sense.\\

For the strongly proximal action, the unique minimal set is described by the following lemma from \cite{BT}\footnote{Therein, the results are stated for the action of semi-group. This leads us to ask whether it is possible to extend the results of this paper to the action of semi-group.	
}.
\begin{lem}\label{BouT}Let $G$ acts on a compact space $X$ and
	$M=$\linebreak $\big\{x \in X/ \forall \mu \in \Pa(X), \delta_x \in \overline{\text{Orb}_G(\mu)}\big\}$. Then, the following are equivalent.
	\begin{enumerate}
		\item The action is strongly proximal.
		\item   $M$ is the unique minimal subset.
		\item  $M$ is nonempty.
	\end{enumerate}
\end{lem}

We need also the following lemmas from \cite{nadler} and \cite{dendron} .

\begin{lem}\label{arc} Let $X$ be a dendrite with metric $d$. Then, for every $\eps > 0$, there is a $\delta > 0$
	such that $diam([x, y]) < \eps$ whenever $d(x, y) < \delta$.
\end{lem}


\begin{lem}[\cite{dendron}]\label{Closed-fa} Let $X$ be a dendron. Then
\begin{enumerate}[(i)]
\item $X$ is locally connected.	
\item If $(Y_i)_{i \in I}$ is a family of subcontinua of $X$ such that for any $i \neq j$,  $Y_i \cap Y_j \neq \emptyset$. Then,
$\ds \bigcap_{i \in I}Y_i \neq \emptyset.$
\item For any $x, y \in X$, $$\ds [x,y]=\bigcap\Big\{C \textrm{~~subcontinuum~~of~~$X$ ~~which~~contains~~} x,y  \Big\}.$$
\end{enumerate}	
\end{lem}

 We notice that $(i)$ and $(ii)$ are stated in \cite{dendron} as Corollary 2.15.1, and the proof of $(iii)$, as noted by Malyutin \cite{Maly}, can be obtained as a consequence of Lemma 2.7 combined with Corollary 2.14 of \cite{dendron}.

\section{main results on equicontinuity}\label{SMain-1}

 We start by stating our first main result.

 \begin{thm}\label{equicon}
 	Let $G$ be a group acting on a dendrite $X$ and assume that the action has a finite orbit. If $M$ is an
 	infinite minimal set of $G$ then the action of $G$ restricted to $M$ is equicontinuous.
 \end{thm}

Theorem \ref{equicon} can be augmented as follows.

\begin{thm}\label{equiconlocalden}
Let $G$ be a group acting on a local dendrite $X$ and assume that the action has a finite orbit. If $M$ is an
infinite minimal set of $G$ then the action of $G$ restricted to $M$ is equicontinuous.
\end{thm}

Our second main result is as follows.

\begin{thm}\label{noamenable-equicon}
There exists a class of non-amenable groups for which the action on dendrite has a finite orbit. Therefore, the restricted action to any minimal set is equicontinuous and infinite minimal sets are Cantor sets.
\end{thm}

For the proof of our main results, we need the following.

\begin{prop}[\cite{M-Nag}]\label{pr4proc} Let $X$ be a dendrite, a group $G$ act on $X$ and $M$ be a minimal set of $G$. Assume that the action of $G$ has a finite orbit. Then,
 \begin{enumerate}[(i)]
\item$M$ is the set of endpoints of $[M]$.
\item For any point $a\in X$ in a finite orbit, the point $r_{[M]}(a)$ is in a finite orbit.
\item $[M]$ contains at least one finite orbit consisting of one or two points.
\item If $M$ is infinite, then it is the only infinite minimal subset in $[M]$.
\end{enumerate}
\end{prop}
Moreover, we have the following.

\begin{prop}[\cite{M-Nag}]\label{h}Let $M$ be a minimal set of $G$. If $M$ is infinite, then there exists a sequence $(T_n)_{n\in\mathbb{N}}$ of sub-trees in $[M]$ satisfying the following properties:
\begin{enumerate}[(i)]
\item $T_n\subset T_{n+1}$, $\forall n\in\mathbb{N}$.
\item $E(T_n)\subset B([M])$ is a finite orbit, $\forall n\in\mathbb{N}$.
\item $[M]=\overline{\bigcup_{n\in\mathbb{N}} T_n}$.
\item $\underset{n\to+\infty}\lim E(T_n) = M$ (in the Hausdorff metric).
\end{enumerate}
\end{prop}
We need also the following results.


\begin{lem}[\cite{M-Nag}]\label{Infinit-mi}Let a group $G$ act on a dendrite $X$. Suppose that the action has a finite orbit.  Then the minimal set $M$ is either a finite orbit or a Cantor set.
\end{lem}

\begin{lem}[\cite{M-Nag}]\label{G-continuum} Let $G$ be a group acting on a continuum $X$ and let $M$ be a minimal set. Then, $M$ is
	\begin{enumerate}
\item a finite orbit, or
\item X, that is, the action is minimal, or
\item a $G$-invariant, compact perfect nowhere dense subset of $X$.		
	\end{enumerate}
\end{lem}
For the case of the dendrites, we need the following results from \cite{M-Nad},
and \cite{Gla-M} which we state in the more general setting. For sake of completeness, we present the proof.

\begin{thm}\label{Shi} Let $G$ acts on a dendron $X$ and assume that the action has no finite orbit. Then
	
\begin{enumerate}
	\item \label{one} There is a unique infinite minimal set $M \subset X.$
	\item \label{two} The action of $G$ on $M$ is strongly proximal.
	\item \label{three} $G$ contains a non-abelian free group on two generators.
\end{enumerate}
\end{thm}

\begin{proof} For the proof  of (\ref{one}),
We proceed by contradiction.
Let $(M_\alpha)_{\alpha \in I}$ be the family of minimal sets of $X$. Then, for any $\alpha, \beta \in I$,  if $M_\alpha \cap M_\beta=\emptyset$ then $[M_\alpha] \cap [M_\beta] \neq \emptyset$, since otherwise the endpoints of the arc $[r_{[M_\alpha]}(x),r_{[M_\beta]}(y)]$ (where $(x,y)$ is any pair in $[M_\beta]\times [M_\alpha]$) would form a finite orbit. Put
$$X_\infty=\bigcap_{\alpha}[M_\alpha].$$
Then, by Lemma \ref{Closed-fa}, we get that $X_\infty$ is nonempty. We further have that the restriction action of $G$ to it has no finite orbit.
Now, let $M_1,M_2$ be two minimal sets of $X_\infty$, then $ [M_1]=[M_2]=X_\infty$, and so $M_1=M_2$. At this stage, we have proved that in $X_\infty$ there is a unique minimal set $M$ and its convex hull $[M]=X_\infty\subset [M_{\alpha}]$ for any $\alpha$.

Fix an $\alpha\in I$ such that $M_\alpha\cap M=\emptyset$, and put $Y=[M_{\alpha}]$. Let $x\in M_\alpha$ and suppose $Z$ is the connected component in $Y$ of $Y\setminus X_\infty$ that contains $x$ and $a=r_{X_\infty}(x)$. It follows that for any $g\in G$, $g(\overline{Z})\cap X_{\infty}=\{g(a)\}$. By assumption, there is no finite orbit, therefore the orbit of $a$ is infinite and so is the orbit of the set $Z$. By taking a suitable net $(g_i)_{i \in I}$ in $G$,  we may assume that $\lim_{i}g_i(a)=b\in X_\infty$ and $g_i(a)\neq b$ for all $i$. Therefore $M$ is a subset of $\overline{\text{Orb}_G(b)}$, and the point $b$ is obviously outside the set $M_\alpha$. But $X$ is rim-finite, we can thus find a neighborhood $U$ of $b$ in $X$ with finite boundary and such that $U\cap M_\alpha=\emptyset$. It turns out that for infinitely of $i$, $\{g_i(a)\}$ meets $U$ however $\{g_i.x\}$ meets $X\setminus U$ and this yields that the boundary of $U$ is infinite, which is inconsistent with the rime-finiteness.\\

Let us denoted by $M$ this unique minimal set. (\ref{two}) follows by the same arguments as in \cite{Gla-M}. Therefore, by Lemma \ref{BouT}, $$M =\big\{x \in X/ \forall \mu \in \Pa(X), \delta_x \in \overline{\text{Orb}_G(\mu)}\big\}.$$ But, the factor of strongly proximal action is strongly proximal. Whence, the action of $G$ on $M$ is strongly proximal. To finish the proof, we need to prove (\ref{three}). For that we follows \cite{Gla-M}. The proof of the lemma is complete.
\end{proof}

This combined with Margulis's theorem \cite{Marg} gives the following:

\begin{cor} Let $G$ be a group acting on a local dendrite $X$ without invariant probability measure. Then,the action is strongly proximal and the group must contain a non-abelian free group.
\end{cor}

\begin{proof}[\textbf{Proof of Theorem \ref{equicon}}]
Suppose that $M$ is an infinite minimal subset. Let $(T_n)_n$ be the sequence of trees defined in Proposition \ref{h}. For each $n\in\mathbb{N}$ and for any $a\in E(T_n)$, let $F^n(a)$ be the closure of the union of all connected components of $[M]\setminus \{a\}$ which are disjoint from $T_n$. In this way, $F^n(a)=\ds \bigcup [a,x]$ where the union is taken over all arcs $[a,x]$ such that $x\in M$ and $[a,x]\cap T_n=\{a\}$. Put $F^n=\{F^n(a): \ a\in E(T_n)\}$.

Claim 1. $\lim_{n\to +\infty} Mesh(F^n)=0$. Let $\eps>0$ and $\delta$ be as in Lemma \ref{arc}. By assertion (iv) in Proposition \ref{h}, there is $N\in\mathbb{N}$ such that $d_H(E(T_n), M)<\delta$ for any $n\geq N$. Take an integer $n\geq N$ and any $a\in E(T_n)$, if $x\in M$ is such that $[a,x]\cap T_n=\{a\}$ then there is $b\in E(T_n)$ such that $d(x,b)<\delta$, it follows from Lemma \ref{arc} that $diam([b,x])<\eps$. It turns out that $d(a,x)<\eps$ since $[a,x]\subset[b,x]$. Furthermore, $\diam(F^n(a))<2\eps$. We notice that an alternative proof can be given using the continuity of the map $F \mapsto [F].$\\

Claim 2. For any $g\in G$, $n\in\mathbb{N}$ and $a\in E(T_n)$, $g(F^n(a))=F^n(g(a))$. Indeed, let $x\in M$ be such that $[a,x]\cap T_n=\{a\}$ then $g([a,x])=[g(a),g(x)]$. Moreover, as $g(T_n)=T_n$, $[g(a),g(x)]\cap T_n=\{g(a)\}$.
 It follows that $g(F^n(a))=F^n(g(a))$.\\

Now, let $\eps>0$ and $x\in M$. Then by claim 1, there is $n\in\mathbb{N}$ such that $Mesh(F^n)<\eps$. Take $a\in E(T_n)$ such that $x\in F^n(a)$ then $F^n(a)$ is a neighborhood of $x$ in $[M]$. For any $y\in F^n(a)$ and for any $g\in G$, we have $g(y)\in F^n(g(a))$. However, $\diam (F^n(g(a))<\eps$. This shows the equicontinuity of the action of $G$ restricted to $M$. 	
\end{proof}

For the proof of Theorem \ref{equiconlocalden}, we need the following two lemmas from \cite{M-Nag2}.

\begin{lem}\label{circle-finite}
	Let $G$ be a subgroup of Hom($S^1$) and $M \subset S^1$ a minimal set of $G$. Assume
	that $G$ has a finite orbit. Then M is a finite orbit.
\end{lem}

\begin{lem}\label{graph-finite}
	Let $X$ be a graph different from a circle and a group $G$ acting on $X$ . Then a
	minimal set $M$ of $G$ is finite (in fact a finite orbit).
\end{lem}

We are now able to proceed to the proof of Theorem \ref{equiconlocalden}.

\begin{proof}[\textbf{Proof of Theorem \ref{equiconlocalden}}]
	Suppose that $X$ is a local dendrite not a dendrite. First, note that by assuming the existence of finite orbit, a finite one will occur in the canonical graph $\Gamma_X$. Indeed, if $\{a_1,\dots,a_s\}$ is finite orbit disjoint from $\Gamma_X$, then for $i=1,\dots,s$, let $Y_i$ be the connected component of $X\setminus\Gamma_X$ that contains $a_i$ and let $\{b_i\}=\overline{Y_i}\cap \Gamma_X$. For any $i,j\in\{1,\dots,s\}$, there is $g\in G$ such that $g(a_i)=a_j$. Hence,
	$g(Y_i)=Y_j$, since, obviously, $\Gamma_X$ is $G$-invariant. It follows that $g(b_i)=b_j$. Therefore $\{b_1,\dots,b_s\}$ is a finite orbit included into $\Gamma_X$. Furthermore, any minimal set with non empty intersection with $\Gamma_X$ will be entirely included into $\Gamma_X$. Hence, it is finite by Lemma \ref{circle-finite} and Lemma \ref{graph-finite}. Now suppose $K$ is an infinite minimal set. Then $K$ is disjoint from $\Gamma_X$ and intersects at most finitely number of connected component of $X\setminus \Gamma_X$, which we denote by $Z_1,\dots,Z_k$. Put $\{z_i\}=Z_i\cap \Gamma_X$. by similar reasoning as above, $\{z_1,\dots,z_k\}$ is a finite orbit. Now let us consider the set $Z=Z_1\cup\dots\cup Z_k$. This set is $G$-invariant and it is a finite union of dendrites. By collapsing the orbit $\{z_1,\dots,z_k\}$ to a single point $z$, we obtain a quotient space $\widetilde{Z}$ which is in fact a dendrite, in this new quotient space, the set $K$ still minimal for the generated action of $G$ and so by Theorem \ref{equicon}, $K$ is equicontinuous for the generated action on $\widetilde{Z}$ and the same holds for the initial action on $Z$.
	
\end{proof}

\section{On the action of non-amenable group and Thompson groups.}\label{SMain-2}
The notion of amenability was introduced by von Neumann to shed some light on the Banach-Tarski paradox \cite{New}. Roughly speaking, the Banach-Tarski paradox \cite{BanaT} say that in $\Real^3$, every two bounded sets $A$ and $B$ with non-empty interior can be decomposed in finite pieces say $n$ such that $\displaystyle  A=\bigcup_{i=1}^{n}A_i$ and $\displaystyle  B=\bigcup_{i=1}^{n}B_i$ so that $A_i$ can be rotated to $B_i$, $i=1,\cdots,n$. As, we will see, von Neumann noticed that this paradox is due to the fact  that the group of rotation $SO(3)$ is not amenable.\\

We recall that if $G$ is a group, then, obviously, $G$ acts from the left on the space of all bounded complex-valued function  $\mathcal{B}(G)$ equipped with the uniform norm $\|.\|_{\infty}$. $G$ is said to be amenable if there exist a non-negative bounded linear functional $\mu$ on $\mathcal{B}(G)$ left invariant and such that $\mu(1)=1$. $\mu$ is called left mean. In the case of discrete group (group equipped with discrete topology) this equivalent to the existence of finitely additive measure.  von Neumann proved that the class of amenable groups is closed under subgroups, factor groups, extensions and  direct groups. He further proved that all abelian groups are amenable. It follows that the solvable groups are amenable. By applying F\o{}lner criterion, it is easy to see that the finitely generated groups of subexponential growth are amenable. It follows that the finitely generated nilpotent groups are amenable. We recall that the group $G$ is amenable if and only it has a F\o{}lner sequence, that is, a sequence $\{F_n\}$ such that
$$\frac{\big|gF_n \Delta F_n\big|}{\big|F_n\big|}\tend{n}{+\infty}0.$$
von Neumann proved also the following Lemma. We present the proofs for sake of completeness of exposition.
\begin{lem}\label{Vn}Let $G$ be a group which contain a free non-abelian group then $G$ is not amenable.
\end{lem}
We recall that the free group $\F_S$ generated by the alphabet $S$ is the set of all the class words generated by $S$. For a nice account we refer to \cite[Chap. 7]{Hall}. Let $s \in S$, we denoted by $W(s)$ the set of reduced words beginning with $s$. Now, let us observe that the proof of Lemma \ref{Vn} follows from the following.

 \begin{lem}\label{Vn1}Let $\F_2$ be a free non-abelian group generated by two element $s,t$  then $\F_2$ is not amenable.
 \end{lem}

\begin{proof} Suppose, \it{per impossible}, that there is  a left invariant finitely additive measure $\mu$ on $\F_2$. We start by noticing that we have
\begin{eqnarray}\label{BT1}
\F_2=\bigcup_{g \in E}W(g), \textrm{~~where~~} E=\big\{e,s,s^{-1},t,t^{-1}\big\}
\end{eqnarray}	
and
\begin{eqnarray}\label{BT2}
\F_2=W(s) \bigcup s^{-1} W(s) = W(t) \bigcup t^{-1} W(t).
\end{eqnarray}
Therefore, $\mu{\{e\}}=0$. Moreover, since $\mu$ is additive, by \eqref{BT1}, we get,
$$\mu(\F_2)=
\mu(\{e\})+\mu(W(s))+\mu(W(s^{-1}))+\mu(W(t))+\mu(W(t^{-1})),$$
We further have, by \eqref{BT2},
$$\mu(\F_2)=\mu(W(s))+\mu(W(s^{-1}))=
\mu(W(t))+\mu(W(t^{-1})).$$ This yields a contradiction since $\mu(\F_2)=1,$ and thus the proof of the theorem is complete.
\end{proof}
Let us notice that the equation \eqref{BT1} combined with \eqref{BT2} is exactly what is called Banach-Tarski paradox.\\

According to Day\cite{Day}, it is seems that von Neumann conjectured that any non-amenable group contain a non-abelian free group. It is turns out that this conjecture is false. Indeed, Ol'shanskii proved that the so-called ``Tarski monsters" groups are not non-amenable groups \cite{Ol},\cite{Ol2}, \cite{Ol3}. Later, Adian showed \cite{Ad} that the free Burnside group of odd exponent $>665$ with at least two generator is non amenable \footnote{The representation of the group is given by $\big<a_1,\cdots,a_m/w^n=1\big>$, where $w$ is in the set of all words of the alphabet$\{a_1,\cdots,a_m\}$}. Very recently, N. Monod \cite{Mono} gives a simple counterexample to \linebreak von Neumann conjecture. Indeed, he proved that if $A$ is a subring of $\Real$ and $A \neq \Z$, then the  subgroups $H(A)$ of the group $G(A)$ of piecewise projective transformations contain finitely generated subgroups that are non-amenable without non-abelian free subgroups.\\

For the proof of Theorem \ref{noamenable-equicon}, we start by stating the following lemma. This lemma can be seen as a corollary of the main result in \cite{Mono}. But, we give here a simple proof.
\begin{lem}\label{orbitfinie}
	Let $G$ be a non-amenable group without non-abelian free group acting on the nondegenerate dendrite $X$. Then the action has a finite orbit.
\end{lem}
\begin{proof} Suppose, \textit{per impossible}, that the action has no finite orbit. Then, by Lemma \ref{Shi}, the action on its minimal set is strongly proximal. We further have that the group $G$ contains a free group which contradicts our assumption. The proof of the lemma is complete.
\end{proof}	

\begin{lem}\label{keylemma1}
	Let $G$ be a non-amenable group without non-abelian free group acting on the nondegenerate dendrite $X$. Then the infinite minimal sets of the action of $G$ are Cantor sets.
\end{lem}
\begin{proof} By Lemma \ref{orbitfinie}, the action has a finite orbit. Therefore, by appealing to Lemma \ref{Infinit-mi}, it follows that the infinite minimal sets are Cantor sets. The proof of the lemma is complete.
\end{proof}

\begin{proof}[\textbf{Proof of Theorem. \ref{noamenable-equicon}}] Take $G$ to be any non-amenable group without non-abelian free group. For example, ``Tarski monsters" groups, free Burnside group of odd exponent $>665$ with at least two generator or the subgroups of $H(A)$. Suppose $X$ is a non-degenerate dendrite then by Lemma \ref{orbitfinie} any action of $G$ on $X$ has a finite orbit. By Theorem \ref{equicon}, the restriction of this action to any minimal subset is equicontinuous. Moreover, by Lemma \ref{keylemma1}, any infinite minimal set is a Cantor set.
\end{proof}

\subsection{Thompson group $F$}
Consider the following two homeomorphisms of $[0,1]$:

\begin{equation*}
f(x) = \begin{cases}\frac{x}2 & \quad\text{if $0 \leq x \leq \frac12$}\\
x-\frac14 & \quad\text{if $\frac12<x <\frac34$}\\
2x-1 & \quad\text{if $\frac34 \leq x \leq 1$}
\end{cases}
,~~
g(x) = \begin{cases}x & \quad\text{if $0 \leq x \leq \frac12$}\\
\frac{x}2+\frac14 & \quad\text{if $\frac12<x <\frac34$ }\\
x-\frac18 & \quad\text{if $\frac34 < x <\frac78 $}\\
2x-1 & \quad\text{if $\frac78 \leq x \leq 1.$}
\end{cases}
\end{equation*}
The group generated by $f$ and $g$ is called Thompson group. We denoted by $F$, so $F=<f,g>$. For a nice account on Thompson groups, we referee to \cite{CFP}.\\
B. Duchesne and N. Monod  established that the action of Thompson group has a finite orbit \cite{Dush-Mono}. Form this we deduce the following.

\begin{cor}\label{main3} Any action of the Thompson group $F$ on a dendrite $X$ restricted to its minimal sets is equicontinuous.
\end{cor}
\begin{rem}It is seems not known whether Thompson group is amen- \linebreak
able. But, according to our main results combined with Corollary \ref{main3}, it can be deduced that the strategy of the study of the dynamical properties of the action of $F$ on dendrite may not be the right tool to highlight the question of amenability of $F$.
\end{rem}

\section{Invariant measures of group action on dendrites.}\label{SMain-3}

We start by pointing out that, by the classical Krylov-Bogolyobov, any amenable group action has a invariant probability measure on dendrite. Indeed, Krylov-Bogolyobov procedure allows us to construct an invariant measure as  limit of sequence of Dirac measures. This later sequence has a weak limit if the space $X$ is compact since the set of the probabilities measures is compact convex space with respect to the weak star topology. The limit is invariant since the group is amenable. We further have by the classical Krein-Milan theorem that the set of extremal measures, that is, ergodic probabilities measures is not empty. We recall that the probability measure $\mu$ is ergodic if the measure of any Borel invariant set is $0$ or $1$. According, to the Choquet representation theorem \cite[p.79]{Rudin}, any invariant measure is a barycenter of ergodic elements of the set of invariant measure. We need here to explore the case of general groups. In this case, such invariant measure may not exist.\\

But, as for the case of the action of the group of homeomorphism on the real line and circle, on may expect to give a characterization of the support of such measure when it is exist. The endpoints will play a great role in this characterization. \\

It follows from the previous discussion that the action of any amenable group has a finite orbit. We thus get that the set of invariant probability measures under the action of any amenable group is not empty since the  probability measure supported uniformly on a finite orbit is in $\Pa(X\looparrowleft G).$




The following result is due essentially to Glasner and Megrelishvili \cite{Gla-M}.

\begin{thm}[\cite{Gla-M}]\label{Gla-M} Let $G$ be an amenable group acting on a dendrite $X$, then, the ergodic invariant measure is either a uniform distribution on a finite set, or the unique ergodic measure on an infinite minimal set.
\end{thm}

Theorem \ref{equiconlocalden} combined with the strict ergodicity of minimal equicontinuous  action \cite[Theorem 6. p.53]{Aus} allows us to strengthen the previous result as follows \footnote{See also Lemma 2 in \cite{Marg}.}.

\begin{thm}
Let $G$ be a group acting on a local dendrite $X$ with finite orbit. Then, the ergodic invariant probability measure is either a uniform distribution on a finite orbit, or it is the uniquely ergodic measure on a minimal infinite set.
\end{thm}

For the case of dendron, we have the following alternative.

\begin{thm} \label{main4}
Let $G$ act on a dendron $X$. Then either the set $\Pa(X\looparrowleft X)$ contains a uniform distribution on a finite set or it is empty.
\end{thm}

\begin{proof} [\textbf{Proof of Theorem \ref{main4}}] Obviously, we have only two alternative, either the action has a finite orbit or not. If the action has a finite orbit. Then,  by taking the probability measure
$$\mu=\frac{1}{\big|M\big|}\sum_{x \in M}\delta_x,$$
where $\delta_x$ is a Dirac measure on $x$, it is easy to see that $\mu$ is an invariant probability measure. Now, suppose that the action has no finite orbit. Then, by Lemma \ref{BouT}, the action is strongly proximal and the minimal set is given by
$$M=\big\{x \in X/ \forall \mu \in \Pa(X), \delta_x \in \overline{\text{Orb}_G(\mu)}\big\}$$
Assume further that there is  a $G$-invariant measure $\mu$. Then, for $x \in M$, we get  $\delta_x=\text{Orbit}_G(\mu)=\mu$. This contradicts our assumption. The proof of the theorem is complete.
\end{proof}

%


We end this section, by producing an effective construction of invariant measure in the case of amenable group action.\\

For the case of invariant measure, we need the following classical result from \cite[Theorem 10.28]{nadler}.
\begin{lem}If $X$ is a nondegenerate dendrite, then X can be written as follows:
	$$X=E(X) \cup \bigcup_{i=0}^{+\infty}A_i,$$
where $E(X)$ is is the endpoint set of $X$ and each $A_i$ is an arc with endpoints $p_i$ and $q_i$ such that
$\ds A_{i+1} \cap  \Big(\bigcup_{j=1}^{i}A_j\Big) = \{p_{i+1}\},$ for each $i = 1 , 2 ,\cdots$
and $\diam (A_i) \longrightarrow 0$ as $i \longrightarrow +\infty.$
\end{lem}
Let $X$ be a dendrite with a metric $\rho$. For any arc $[a,b]$ in $X$, put
$$d(a,b)=\sum_{i=1}^{+\infty}\frac{1}{2^i}\big|h_i^{-1}\big([a,b] \bigcap A_i\big)\big|,$$
Where $(h_i)$ is a family a homeomorphism defined for each $i$ from $[0,1]$ to $A_i$, $|.|$ denote the Euclidean metric on the interval $[0,1]$.\\
 Then, it is observed in \cite{Shi} and it is easy to see, that $d$ is equivalent to $\rho$.\\

In the same manner, we define a probability measure $\mu$ on $X$ by putting, for any $f \in C(X)$,
\begin{eqnarray}\label{mu-lamesure}
\mu(f)&=&  \sum_{i=1}^{+\infty}\frac{1}{2^i}\int_{h_i^{-1}\big(A_i\big)}f \circ h_i(x) dx\\
&=&
\sum_{i=1}^{+\infty}\frac{1}{2^i}\int_{0}^{1}f \circ h_i(x) dx,\nonumber
\end{eqnarray}
$dx$ denote the Lebesgue measure on the interval $[0,1]$. We further have $\mu\big([a,b]\big)=d(a,b).$\\

\begin{thm}\label{continuous} Let $G$ be an amenable group, and suppose that $G$ acts continuously on a non-degenerate dendrite $X$. Then the compact set  $\Pa(X\looparrowleft G)$ contain an invariant probability measure associated to $\mu$.
\end{thm}
\begin{proof}[\textbf{Proof of Theorem \ref{continuous}}]By our assumption, let $\nu$ be a mean on $G$ and $\mu$ the measure defined by \eqref{mu-lamesure}. We further assume without loss of generality, that $G$ is a discrete group.  Put $\tau_g(f)(x)=f(g.x)$, for $f \in \Co(X)$, and $\tau_h(F)(g)=F(g.h),$ for $F \in \mathcal{B}(G)$. Therefore,
	for any $f \in \Co(X)$ and $g \in G$, let us put
\begin{eqnarray*}
	\big(\Phi(f)\big)(g)&=&\int f(g.x) d\mu(x).
\end{eqnarray*}
Obviously $\Phi$ is a positive linear application from $\Co(X)$ to $\mathcal{B}(G)$, that is,
\begin{itemize}
\item $\Phi(1_X)=1_G$,
\item $\Phi(f) \geq 0$ if $f \geq 0$, and
\item $\Phi(f+c.g)=\Phi(f)+c.\Phi(g),$, $f,g \in C(X)$ and $c \in \C.$
\end{itemize}
This, combined with the Riesz representation theorem, allows us to define a positive probability measure by putting,
$$\widetilde{\mu}(f)=\nu(\Phi(f)).$$ We further have
\begin{eqnarray*}
g\widetilde{\mu}(f)&=&\widetilde{\mu}\big(\tau_g(f)\big)\\
&=&\nu(\Phi(\tau_g(f))).
\end{eqnarray*}
But,
\begin{eqnarray*}
\Phi(\tau_g(f))(h)&=&\int_X \tau_g(f)(h.x)d\mu(x),\\
&=& \int_X f(gh.x) d\mu(x)\\
&=&\Phi(f)(gh).
\end{eqnarray*}
Whence
\begin{eqnarray*}
g\widetilde{\mu}(f)&=&\int_G\Phi(f)(gh)d\nu(h)\\
&=&\int_G \Phi(f)(h) d\nu\\
&=&\widetilde{\mu}(f),
\end{eqnarray*}
since, for any bounded function $F : G \longrightarrow \C$, we have
$$\int_{G}F(g.h)d\nu(h)=\int_{G}F(h)d\nu(h).$$
The proof of the theorem is complete.	
\end{proof}

\begin{rem}In the proof of Theorem \ref{continuous}, we assume that $G$ is a discrete group. Therefore, the mean is defined on $\mathcal{B}(G)$. For the more general case of locally compact group $G$, the mean is defined on $L^{\infty}(G)$ or $\Co B(G)$  the space of bounded continuous functions.
\end{rem}
At this point, let us further observe that we have the following result which we hope can be used in future studies and research.

\begin{thm}\label{Precurrent}Let $G$ be a group acting on the non-degenerate local dendrite $X$ and assume that the set of recurrent points is not empty. Then there exists a continuous probability measure which gives one or zero to any invariant Borel set.
\end{thm}
\begin{proof}By Glimm-Effros theorem (see \cite[p. 91]{Nad}), if $x$ is a recurrent point then there is a continuous probability measure  on $\overline{\textrm{Orb}_{G}(x)}$ which gives mass zero or one to $G-$invariant sets. Since the set of recurrent points is not empty, the result follows.
\end{proof}

\begin{cor} Let $G$ acts on the non-degenerate local dendrite $X$. Suppose that $G$ is an amenable group and the set of recurrent points is not empty. Then, there exists an invariant probability measure which is ergodic.
\end{cor}
\begin{proof}By Theorem \ref{Precurrent}, there exists a continuous probability measure with gives one or zero to any invariant Borel set. Let $\sigma$ be such probability measure. By applying the same procedure as in the proof of Theorem \ref{continuous}, we get an invariant probability measure which is ergodic. This finish the proof of the corollary.
\end{proof}

\textbf{Questions.}
\begin{enumerate}
\item Let $G$ be a semi-group and $X$ a local dendron.  It is natural to ask if Theorem \ref{equiconlocalden} can be extended to the action of $G$ to $X$.
\item Let $G$ be a group and $X$ a dendrite. One may ask on the classification of invariant $\sigma$-finite measures, ergodic invariant $\sigma$-finite measures, stationary measures (Quasi-invariant probability measures) for amenable and non-amenable group and their ergodic proprieties.
\end{enumerate}
\vskip 1em

Let us point out that a characterization of discrete spectrum for action of amenable group on compact space is given in \cite{el-abd-Mahesh}. We notice here that according to Theorem \ref{equiconlocalden} the restriction action to any minimal of local dendrite has a topological discrete spectrum.\\

\textit{Acknowledgements.} The first author would like to express his heartfelt thanks to Ahmed Bouziad, Mahesh Nerurkar, Enhui Shi and XiangDong Ye for valuable discussions on the subject. He would also like to thank Habib Marzougui and University of Carthage, Faculty of Science of Bizerte (where this paper was started) for the invitation and for their hospitality. The authors would like to thank H. Marzougui for valuable discussions.\\

The second author would like to thanks university of Rouen Normandy where the paper was revised and augmented, for the invitation and hospitality. \\

The work of the second author was supported by the research unit:
``Dynamical systems and their applications'' (UR17ES21), Ministry of Higher Education and Scientific Research,
Tunisia.

\bibliographystyle{amsplain}

\end{document}